\documentclass[11pt]{amsart}
\usepackage{amsmath,amsthm,amssymb, amsfonts, amscd}
\usepackage{enumerate}
\usepackage{color}

\usepackage{braket}
\newcommand{\pctext}[2]{\text{\parbox{#1}{#2}}}

\newtheorem{thm}{Theorem}[section]
\newtheorem{cor}[thm]{Corollary}

\newtheorem{conj}[thm]{Conjecture}

\theoremstyle{plain}
\numberwithin{equation}{thm}
\theoremstyle{remark}

\def\P{\mathbb P}
\def\Q{\mathbb Q}

\def\Z{\mathbb Z}

\title{Backward Orbit Conjecture for Latt\`es Maps}
\author{Vijay  Sookdeo}
\address{
Vijay Sookdeo\\
Department of Mathematics\\
The Catholic University of America\\
Washington, DC 20064 \\
}

\begin{document}

\begin{abstract}
For a Latt\`es map $\phi:\P^1 \to \P^1$ defined over a number field $K$, we
prove a conjecture on the integrality of points in the backward orbit of $P\in
\P^1(\overline K)$ under $\phi$.
\\
\\
\emph{Accepted for publication in the Turkish Journal of Analysis
and Number Theory}
\end{abstract}

\maketitle

\section{Introduction}
Let $\phi:\P^1 \to \P^1$ be a rational map of degree $\ge 2$ defined over a
number field $K$, and write $\phi^n$ for the $n$th iterate of $\phi$. For a
point $P\in \P^1$, let $\phi^+(P)=\{P, \phi(P), \phi^2(P), \dots \}$ be the
\emph{forward orbit} of $P$ under $\phi$, and let $$\phi^-(P) = \bigcup_{n\ge 0}
\phi^{-n}(P)$$ be the \emph{backward orbit} of $P$ under $\phi$.  We say $P$ is
$\phi$-\emph{preperiodic} if and only if $\phi^+(P)$ is finite.

Viewing the projective line $\P^1$ as $\mathbb A^1 \cup \{\infty \}$ and
taking $P\in \mathbb A^1(K)$, a theorem of Silverman \cite{sil} states that if
$\infty$ is not a fixed point for $\phi^2$, then $\phi^+(P)$
contains at most finitely many points in $\mathcal{O}_K$, the ring of algebraic
integers in $K$.  If $S$ is the set of all archimedean places for $K$, then
$\mathcal{O}_K$ is the set of points in $\P^1(K)$ which are $S$-integral
relative to $\infty$ (see section 2). Replacing $\infty$ with any point $Q\in
\P^1(K)$ and $S$ with any finite set of places containing all the archimedean
places, Silverman's Theorem can be stated as:  If $Q$ is not a fixed point for
$\phi^2$, then $\phi^+(P)$ contains at most finitely many points which are
$S$-integral relative to $Q$.

A conjecture for finiteness of integral points in backward orbits
was stated in \cite[Conj. 1.2]{sook}.

\begin{conj}\label{conj}  If $Q\in \P^1(\overline K)$ is not $\phi$-preperiodic,
then $\phi^-(P)$ contains at most finitely many points in $\P^1(\overline K)$
which are $S$-integral relative to $Q$.
\end{conj}

In \cite{sook}, Conjecture \ref{conj} was shown true for the powering map
$\phi(z)=z^d$ with degree $d\ge 2$, and consequently for Chebyschev
polynomials. A generalized version of this conjecture, which is stated over a
dynamical family of maps $[\varphi]$, is given in \cite[Sec. 4]{grant_ih}.
Along those lines, our goal is to prove a general form of Conjecture \ref{conj}
where $[\varphi]$ is the family of Latt\`es maps associated to a fixed elliptic
curve $E$ defined over $K$ (see Section \ref{main}).

\section{The Chordal Metric and Integrality}
\subsection{The Chordal Metric on $\mathbb P^N$}\label{chordal}
Let $M_{K}$ be the set of places on $K$ normalized so that the product formula
holds: for all $\alpha\in K^*$, $$\prod_{v\in M_K}|\alpha|_v = 1.$$ For points
$P=[x_0:x_1:\cdots:x_N]$ and  $Q=[y_0:y_1:\cdots:y_N]$ in $\mathbb
P^N(\overline{K}_v)$, define the \emph{$v$-adic chordal metric} as $$\Delta_v
(P,Q)= \frac{\max_{i,j}(|x_iy_j-x_jy_i|_v)}{\max_i(|x_i|_v)\cdot
\max_i(|y_i|_v)}.$$ Note that $\Delta_v$ is independent of choice of projective
coordinates for $P$ and $Q$, and $0\le \Delta_v(\cdot, \cdot) \le 1$ (see
\cite{ShuSil}).

\subsection{Integrality on Projective Curves}\label{integrality}
Let $C$ be an irreducible curve in $\mathbb P^N$ defined over $K$ and $S$ a
finite subset of $M_K$ which includes all the archimedean places.  A
\emph{divisor} on $C$ defined over $\overline{K}$ is a finite formal sum $\sum
n_i Q_i$ with $n_i\in \mathbb Z$ and $Q_i\in C(\overline K)$.  The
divisor is \emph{effective} if $n_i > 0$ for each $i$, and its
\emph{support} is the set $\mbox{Supp}(D)=\{Q_1,\dots, Q_\ell \}$.

Let $\lambda_{Q,v}(P) = -\log \Delta_v(P,Q)$ and $\lambda_{D, v}(P)= \sum
n_i\lambda_{Q_i,v}(P)$ when $D=\sum n_i Q_i$.  This makes $\lambda_{D,v}$ an
arithmetic distance function on $C$ (see \cite{sil2}) and as with any arithmetic
distance function, we may use it to classify the integral points on $C$.

For an effective divisor $D = \sum n_i Q_i$ on $C$ defined over $\overline{K}$,
we say $P \in C(\overline{K})$ is \emph{$S$-integral} relative to $D$, or $P$ is
a $(D, S)$-integral point, if and only if $\lambda_{Q_i^\sigma,v}(P^\tau) = 0$
for all embeddings $\sigma, \tau:\overline{K}\to \overline{K}$ and for all
places $v\not\in S$. Furthermore, we say the set $\mathcal{R}\subset C(\overline
K)$ is $S$-integral relative to $D$ if and only if each point in $\mathcal{R}$
is $S$-integral relative to $D$.

As an example, let $C$ be the projective line $\mathbb A^1 \cup \{\infty\}$,
$S$ be the archimedean place of $K=\Q$, and $D=\infty$.  For $P=x/y$, with $x$
and $y$ are relatively prime in $\Z$, we have $\lambda_{D, v}(P)=-\log|y|_v$ for
each prime $v$.  Therefore, $P$ is $S$-integral relative to $D$ if and only
if $y=\pm 1$; that is, $P$ is $S$-integral relative to $D$ is and only if $P\in
\mathbb Z$.

From the definition we find that if $S_1 \subset S_2$ are finite subsets of
$M_K$ which contains all the archimedean places, then $P$ is a $(D,
S_2)$-integral point implies that $P$ is a $(D, S_1)$-integral point. Similarly,
if $\mbox{Supp}(D_1) \subset \mbox{Supp}(D_2)$, then $P$ is a $(D_2,
S)$-integral point implies that $P$ is also a $(D_2, S)$-integral
point.  Therefore enlarging $S$ or $\mbox{Supp}(D)$ only enlarges the set of
$(D, S)$-integrals points on $C(\overline{K})$.

For $\phi:C_1\to C_2$, a finite morphism between projective curves and $P\in
C_2$, write $$\phi^*P= \sum_ {Q\in \phi^{-1}(P)} e_{\phi}(Q)\cdot Q$$
where $e_\phi(Q) \ge 1$ is the ramification index of $\phi$ at $Q$.
Furthermore, if $D=\sum n_iQ_i$ is a divisor on $C$, then we define
$\phi^*D=\sum  n_i\phi^*Q_i$.

\begin{thm}[Distribution Relation]\label{dist}
Let $\phi:C_1\to C_2$ be a finite morphism between irreducibly smooth curves
in $\mathbb P^N(\overline{K})$.  Then for $Q\in C_1$, there is a finite set of
places $S$, depending only on $\phi$ and containing all the archimedean
places, such that  $\lambda_{P, v} \circ \phi = \lambda_{\phi^*P, v}$
for all $v\not\in S$.
\end{thm}

\begin{proof} See \cite[Prop. 6.2b]{sil2} and note that for projective varieties
the $\lambda_{\delta W \times V}$ term is not required, and that the big-O
constant is an $M_K$-bounded constant not depending on $P$ and $Q$.
\end{proof}

\begin{cor}\label{dist2}
 Let $\phi:C_1\to C_2$ be a finite morphism between irreducibly smooth curves
in $\mathbb P^N(\overline{K})$, let $P\in C_1(\overline K)$, and let $D$ be an
effective divisor on $C _2$ defined over $K$.  Then there is a finite set of
places $S$, depending only on $\phi$ and containing all the archimedean
places, such that $\phi(P)$ is $S$-integral relative to $D$ if and only
$P$ is S-integral relative to $\phi^*D$.
\end{cor}

\begin{proof}
 Extend $S$ so that the conclusion of Theorem \ref{dist} holds.  Then
for $D=\sum n_i Q_i$ with each $n_i > 0$ and $Q_i\in C_2(\overline K)$, we
have that $$\lambda_{\phi^*D,v}(P)=\lambda_{D,v}(\phi(P)) = \sum n_i
\lambda_{Q_i, v}(\phi (P)).$$ So $\lambda_{\phi^*D,v}(P)=0$ if and only if
$\lambda_{Q_i, v}(\phi (P))=0.$
\end{proof}

\section{Main Result}\label{main}
Let $E$ be an elliptic curve, $\psi: E\to E$ a morphism, and $\pi:E \to
\P^1$ be a finite covering.  A \emph{Latt\`es map} is a rational map $\phi:
\P^1 \to \P^1$ making the following diagram commute:
$$
\begin{CD}
E @>\psi>> E \\
@VV \pi V  @VV \pi V \\
\P^1 @>\phi>> \P^1
\end{CD}\medskip
$$
For instance, if $E$ is defined by the Weierstrass equation $y^2=x^3+ax^2+bx+c$,
$\psi=[2]$ is the multiplication-by-2 endomorphism on $E$, and $\pi(x,y)=x$,
then $$\phi(x)=\frac{x^4-2bx^2-8cx+b^2-4ac}{4x^3+4ax^2+4bx+4c}.$$

Fix an elliptic curve $E$ defined over a number field $K$, and for
$P\in\P^1(\overline K)$ define:
\begin{center}
\begin{align*}
[\varphi] &= \Biggl\{\phi:\P^1 \to \P^1   \;\bigg|\;
\pctext{2.5in}{there exist $K$-morphism $\psi:E\to E$ and finite covering
$\pi:E\to \P^1$ such that
$\pi\circ \psi = \phi \circ \pi$}\;\Biggr\} \\ \\
\Gamma_0 &= \bigcup_{\phi\in[\varphi]}\phi^+(P)\\ \\
\Gamma&= \left( \bigcup_{\phi\in[\varphi]} \phi^-(\Gamma_0) \right) \cup
\P^1(\overline K)_{[\varphi]-\mbox{preper}}
\end{align*}
\end{center}

A point $Q$ is $[\varphi]$-preperiodic if and only if $Q$ is $\phi$-preperiodic
for some $\phi\in[\varphi]$.  We write $\P^1(\overline
K)_{[\varphi]-\mbox{preper}}$ for the set of $[\varphi]$-preperiodic points in
$\P^1(\overline K)$.

\begin{thm}
If $Q\in \P^1(\overline K)$ is not $[\varphi]$-periodic, then $\Gamma$ contains
at most finitely many points in $\P^1(\overline K)$ which are $S$-integral
relative to $Q$.
\end{thm}

\begin{proof}
Let $\Gamma_0'$ be the $\mbox{End}($E$)$-submodule of $E(\overline K)$ that is
finitely generated by the points in $\pi^{-1}(P)$, and let $$\Gamma'=\{\xi\in
E(\overline K) \mid \lambda(\xi) \in \Gamma_0' \mbox{ for some non-zero }
\lambda \in \mbox{End}(E) \}.$$  Then $\pi^{-1}(\Gamma)\subset \Gamma'$.
Indeed, if $\pi(\xi) \in \Gamma$ is not $[\varphi]$-preperiodic, then $\xi$ is
non torsion and $(\phi_1\circ \pi)(\xi)\in \Gamma_0$ for some Latt\`es map
$\phi_1$.  So $(\pi\circ \psi_1)(\xi) \in \Gamma_0$ for some morphism
$\psi_1:E\to E$, and this gives $(\pi\circ \psi_1)(\xi)= \phi_2(P)$ for some
Latt\`es map $\phi_2$. Therefore $\psi_1(\xi) \in (\pi^{-1}\circ \phi_2)(P) =
(\psi_2 \circ \pi^{-1})(P)$ for some morphism $\psi_2:E\to E$.  Since any
morphism $\psi:E\to E$ is of the form $\psi(X)=\alpha(X)+T$ where $\alpha\in
\mbox{End}(E)$ and $T\in E_{\mbox{tors}}$ (see \cite[6.19]{sil3}), we find
that there is a $\lambda \in \mbox{End}(E)$ such that $\lambda(\xi)$ is in
$\Gamma_0'$, the $\mbox{End}($E$)$-submodule generated by $\pi^{-1}(P)$.
Otherwise, if $\pi(\xi)\in \Gamma$ is $[\varphi]$-preperiodic, then
$\pi\left(E(\overline K)_{\mbox{tors}}\right) = \P^1(\overline
K)_{[\varphi]-\mbox{preper}}$ (\cite[Prop. 6.44]{sil3}) gives that $\xi$ may be
a torsion point; again $\xi \in\Gamma'$ since $E(\overline K)_{\mbox{tors}}
\subset \Gamma'$.  Hence $\pi^{-1}(\Gamma) \subset \Gamma'$.

Let $D$ be an effective divisor whose support lies entirely in $\pi^{-1}(Q)$,
let $\mathcal{R}_Q$ be the set of points in $\Gamma$ which are $S$-integral
relative to $Q$, and let $\mathcal{R}'_D$ be the set of points in $\Gamma'$
which are $S$-integral relative to $D$.  Extending $S$ so that Theorem
\ref{dist} holds for the map $\pi:E\to \P^1$, and since  $\mbox{Supp}(D)\subset
\mbox{Supp}(\pi^*Q)$, we have: if $\gamma \in \Gamma$ is $S$-integral relative
to $Q$, then $\pi^{-1}(\gamma)$ is $S$-integral relative to $D$. Therefore
$\pi^{-1}(\mathcal{R}_Q) \subset \mathcal{R}'_{D}$.  Now $\pi$ is a finite map
and $\pi(E(\overline K)) = \P^1(\overline K)$; so to complete the proof, it
suffices to show that $D$ can be chosen so that $\mathcal{R}'_{D}$ is finite.

From \cite[Prop. 6.37]{sil3}, we find that if $\Lambda$ is a nontrivial
subgroup of $\mbox{Aut}(E)$, then $E/{\Lambda} \cong \P^1$ and the map $\pi:E
\to \P^1$ can be determine explicitly.  The four possibilities for
$\pi$, which are $\pi(x,y) = x,\, x^2,\, x^3$, or $y$ correspond respectively
to the four possibilities for $\Lambda$, which are $\Lambda = \mu_2, \, \mu_4,
\, \mu_6$, or $\mu_3$, which in turn depends only on the $j$-invariant of $E$.
(Here, $\mu_N$ denotes the $N$th roots of unity in $\mathbb C$.)

First assume that $\pi(x,y)\not=y$.  Since $Q$ is not $[\varphi]$-preperiodic,
take $\xi \in \pi^{-1}(Q)$ to be non-torsion. Then $-\xi \in \pi^{-1}(Q)$ since
$\Lambda = \mu_2, \, \mu_4$, or $\mu_6$, and $\xi - (-\xi) = 2\xi$ is
non-torsion.  Taking $D=(\xi)+(-\xi)$, \cite[Thm. 3.9(i)]{grant_ih} gives that
$\mathcal{R}'_{D}$ is finite.

Suppose that $\pi(x,y)=y$.  Then $\pi^{-1}(Q) = \{\xi, \xi', \xi''\}$ where
$\xi+\xi'+\xi''=0$ and $\xi$ is non-torsion since $Q$ is not
$[\varphi]$-preperiodic.  Assuming that both $\xi-\xi'$ and $\xi-\xi''$ are
torsion give that $3\xi$ is torsion, and this contradicts the fact that $\xi$
is non-torsion.  Therefore, we may assume that $\xi-\xi'$ is non-torsion.
Now taking $D=(\xi)+(\xi')$, \cite[Thm. 3.9(i)]{grant_ih} again gives that
$\mathcal{R}'_{D}$ is finite. Hence $\mathcal{R}_Q$, the set of points in
$\Gamma$ which are $S$-integral relative to $Q$, is finite.
\end{proof}

\bibliographystyle{amsalpha}
\bibliography{back_orb}

\providecommand{\bysame}{\leavevmode\hbox to3em{\hrulefill}\thinspace}
\providecommand{\MR}{\relax\ifhmode\unskip\space\fi MR }
% \MRhref is called by the amsart/book/proc definition of \MR.
\providecommand{\MRhref}[2]{%
  \href{http://www.ams.org/mathscinet-getitem?mr=#1}{#2}
}
\providecommand{\href}[2]{#2}
\begin{thebibliography}{Soo11}

\bibitem[GI13]{grant_ih}
David Grant and Su-Ion Ih, \emph{Integral division points on curves},
  Compositio Mathematica \textbf{149} (2013), no.~12, 2011--2035.

\bibitem[KS09]{ShuSil}
Shu Kawaguchi and J.~H. Silverman, \emph{Nonarchimedean green functions and
  dynamics on projective space}, Mathematische Zeitschrift \textbf{262} (2009),
  no.~1, 173--197.

\bibitem[Sil87]{sil2}
J.~H. Silverman, \emph{Arithmetic distance functions and height functions in
  diophantine geometry}, Mathematische Annalen \textbf{279} (1987), no.~2,
  193--216.

\bibitem[Sil93]{sil}
\bysame, \emph{Integer points, {D}iophantine approximation, and iteration of
  rational maps}, Duke Math. J. \textbf{71} (1993), no.~3, 793--829.

\bibitem[Sil07]{sil3}
\bysame, \emph{The arithmetic of dynamical systems}, Graduate Text in
  Mathematics 241, Springer, New York, 2007.

\bibitem[Soo11]{sook}
V.~A. Sookdeo, \emph{Integer points in backward orbits}, J. Number Theory
  \textbf{131} (2011), no.~7, 1229--1239.

\end{thebibliography}

\end{document}